\documentclass[10pt,a4paper]{article}
    \usepackage{amsmath,,amssymb,amsthm,mathtools,mathrsfs}
    \usepackage{enumerate,enumitem}
    \usepackage[all]{xy}
    \usepackage[margin=25mm]{geometry}
    \usepackage{titlesec}
        \titleformat{\section}{\normalfont\large\bf}{\thesection.}{1ex}{\centering}
        \titleformat{\subsection}[runin]{\normalfont\bf}{\thesubsection.}{1ex}{}[.]
    \theoremstyle{plain}
        \newtheorem{theorem}{Theorem}[section]
        \newtheorem{lemma}[theorem]{Lemma}
        \newtheorem{proposition}[theorem]{Proposition}
        
    \theoremstyle{definition}
        \newtheorem{definition}[theorem]{Definition}
    \theoremstyle{remark}
        \newtheorem{remark}[theorem]{Remark}
    \DeclareMathOperator{\supp}{supp}
    
    \numberwithin{equation}{section}
\begin{document}
\begin{center}\large
    {\bf  Navier-Stokes blow-up rates in certain Besov spaces \\ whose regularity exceeds the critical value by
     $\boldsymbol{\epsilon\in[1,2]}$} \\\ \\
    Joseph P.\ Davies\footnote{University of Sussex, Brighton, UK, {\em jd535@sussex.ac.uk}} and Gabriel S. Koch\footnote{University of Sussex, Brighton, UK, {\em g.koch@sussex.ac.uk}}
\end{center}
\ \\
\abstract{For a solution $u$ to the Navier-Stokes equations in spatial dimension $n\geq3$ which blows up at a finite time $T>0$, we prove the blowup estimate ${\|u(t)\|}_{\dot{B}_{p,q}^{s_{p}+\epsilon}(\mathbb{R}^n)}\gtrsim_{\varphi,\epsilon,(p\vee q\vee 2)}{(T-t)}^{-\epsilon/2}$ for all $\epsilon\in[1,2)$ and $p,q\in[1,\frac{n}{2-\epsilon})$, where $s_{p}:=-1+\frac{n}{p}$ is the scaling-critical regularity, and $\varphi$ is the cutoff function used to define the Littlewood-Paley projections. For $\epsilon =2$, we prove the same type of estimate but only for $q=1$: ${\|u(t)\|}_{\dot{B}_{p,1}^{s_{p}+2}(\mathbb{R}^n)}\gtrsim_{\varphi,(p\vee 2)}{(T-t)}^{-1}$ for all $p\in [1,\infty)$. Under the additional restriction that $p,q\in[1,2]$ and $n=3$, these blowup estimates are implied by those first proved by Robinson, Sadowski and Silva (J. Math. Phys., 2012) for $p=q=2$ in the case $\epsilon\in(1,2)$, and by McCormick, Olson, Robinson, Rodrigo, Vidal-L\'{o}pez and Zhou (SIAM J. Math. Anal., 2016) for $p=2$ in the cases $(\epsilon,q)=(1,2)$ and $(\epsilon,q)=(2,1)$.}
\section{Introduction}
According to the physical theory, if an incompressible viscous Newtonian fluid occupies the whole space $\mathbb{R}^{n}$ in the absence of external forces, then the velocity $U(\tau,x)$ and kinematic pressure $\varPi(\tau,x)$ of the fluid at time $\tau>0$ and position $x\in\mathbb{R}^{n}$ satisfy the Navier-Stokes equations
\begin{equation*}
    \left\{\begin{array}{l}
        \partial_{\tau}U-\nu\Delta U+(U\cdot\nabla)U+\nabla\varPi=0, \\ \nabla\cdot U=0,
    \end{array}\right.
\end{equation*}
where the coefficient $\nu>0$ is the kinematic viscosity of the fluid. By considering the rescaled quantities
\begin{equation*}
    t=\nu\tau, \quad u(t,x)=\nu^{-1}U(\nu^{-1}t,x), \quad \varpi(t,x)=\nu^{-2}\varPi(\nu^{-1}t,x),
\end{equation*}
whose physical dimensions are powers of length alone, we may rewrite the Navier-Stokes equations in the standardised form
\begin{equation}\label{intro-navier-stokes}
    \left\{\begin{array}{l}
        \partial_{t}u-\Delta u+(u\cdot\nabla)u+\nabla\varpi=0, \\ \nabla\cdot u=0.
    \end{array}\right.
\end{equation}
At the formal level, if $(u,\varpi)$ satisfy \eqref{intro-navier-stokes} then $\varpi$ may be recovered from $u$ by the formula
\begin{equation}\label{pressure}
    \varpi = {(-\Delta)}^{-1}\nabla^{2}:(u\otimes u).
\end{equation}
Writing $\Lambda$ to denote the physical dimension of length, the quantities $x,t,u,\varpi$ have physical dimensions
\begin{equation*}
    [x] = \Lambda, \quad [t] = \Lambda^{2}, \quad [u] = \Lambda^{-1}, \quad [\varpi] = \Lambda^{-2};
\end{equation*}
this is related to the fact that the standardised Navier-Stokes equations are preserved under the rescaling
\begin{equation}\label{navier-stokes-scaling}
    x_{\lambda}=\lambda x, \quad t_{\lambda} = \lambda^{2}t, \quad u_{\lambda}(t_{\lambda},x_{\lambda}) = \lambda^{-1}u(t,x), \quad \varpi_{\lambda}(t_{\lambda},x_{\lambda})=\lambda^{-2}\varpi(t,x).
\end{equation}
\begin{definition}\label{regsoldef}
    For $T\in(0,\infty]$, a function $u:(0,T)\times\mathbb{R}^{n}\rightarrow\mathbb{R}^{n}$ is said to be a {\em regular solution} to the standardised Navier-Stokes equations on $(0,T)$ if
    \begin{enumerate}[label=(\roman*)]
        \item $u$ is smooth on $(0,T)\times\mathbb{R}^{n}$, with every derivative belonging to $C((0,T);L^{2}(\mathbb{R}^{n}))$;
        \item $(u,\varpi)$ satisfy \eqref{intro-navier-stokes}, with $\varpi$ being given by \eqref{pressure}.
    \end{enumerate}
\end{definition}
\begin{definition}
    For $T\in(0,\infty)$, a regular solution $u$ on $(0,T)$ is said to {\em blow up} at (rescaled) time $T$ if $u$ doesn't extend to a regular solution on $(0,T')$ for any $T'>T$.
\end{definition}
In spatial dimension $n\geq3$, it remains unknown whether there exist regular solutions which blow up. Since Leray's seminal paper \cite{leray1934}, prospective blowing-up solutions have been studied using homogeneous norms:
\begin{definition}
    A norm ${\|\cdot\|}_{X}$ (defined on a subspace $\mathcal{X}\subseteq\mathcal{S}'(\mathbb{R}^{n})$ which is closed under dilations of $\mathbb{R}^{n}$) is said to be {\em homogeneous of degree $\alpha=\alpha(X)$} if, under the rescaling $x_{\lambda}=\lambda x$ and $f_{\lambda}(x_{\lambda})=f(x)$, we have ${\|f_{\lambda}\|}_{X}\approx_{X}\lambda^{\alpha}{\|f\|}_{X}$ for all $\lambda\in(0,\infty)$ and $f\in\mathcal{X}$.
\end{definition}
If ${\|\cdot\|}_{X}$ is homogeneous of degree $\alpha$, then under the scaling \eqref{navier-stokes-scaling} of the Navier-Stokes equations we have ${\|u_{\lambda}(t_{\lambda})\|}_{X}\approx_{X}\lambda^{\alpha-1}{\|u(t)\|}_{X}$, so the quantity ${\|u\|}_{X}$ has physical dimension $\Lambda^{\alpha-1}$. In the context of the Navier-Stokes equations, the homogeneous norm ${\|\cdot\|}_{X}$ is said to be {\em subcritical} if $\alpha(X)<1$, {\em critical} if $\alpha(X)=1$, and {\em supercritical} if $\alpha(X)>1$. If ${\|\cdot\|}_{X}$ is subcritical, then the blowup estimate
\begin{equation}\label{abstract-blowup}
    u\text{ blows up at time }T\quad\Rightarrow\quad{\|u(t)\|}_{X}\gtrsim_{X}{(T-t)}^{-\left(1-\alpha(X)\right)/2}
\end{equation}
makes dimensional sense.

We investigate \eqref{abstract-blowup} in the context of the homogeneous Besov norms
\begin{equation*}
    {\|f\|}_{\dot{B}_{p,q}^{s}(\mathbb{R}^{n})} := {\left\|j\mapsto2^{js}{\left\|\mathcal{F}^{-1}\varphi(2^{-j}\xi)\mathcal{F}f\right\|}_{L^{p}(\mathbb{R}^{n})}\right\|}_{l^{q}(\mathbb{Z})} \quad \text{for }s\in\mathbb{R},\,p,q\in[1,\infty],
\end{equation*}
where $\mathcal{F}$ is the Fourier transform and $\varphi$ is a cutoff function satisfying certain properties.\footnote{We will give more detailed definitions in section \ref{besov-spaces}. Choosing a different function $\varphi$ yields an equivalent norm.} Amongst other things, for any $p, p_j,q,q_j \in [1,\infty]$ and $\delta,s \in \mathbb{R}$ the Besov norms satisfy
\begin{equation}\label{intro-besov-embedding}
    {\|f\|}_{\dot{B}_{p_{1},q_{1}}^{\frac{n}{p_{1}}+\delta}(\mathbb{R}^{n})} \gtrsim_{n} {\|f\|}_{\dot{B}_{p_{2},q_{2}}^{\frac{n}{p_{2}}+\delta}(\mathbb{R}^{n})} \quad \text{if }p_{1}\leq p_{2},\,q_{1}\leq q_{2},
\end{equation}
\begin{equation}\label{intro-lebesgue}
    {\|f\|}_{\dot{B}_{p,1}^{0}(\mathbb{R}^{n})} \geq {\|f\|}_{L^{p}(\mathbb{R}^{n})} \gtrsim_{\varphi} {\|f\|}_{\dot{B}_{p,\infty}^{0}(\mathbb{R}^{n})},
\end{equation}
\begin{equation}\label{intro-sobolev}
    {\|f\|}_{\dot{B}_{2,2}^{s}(\mathbb{R}^{n})} \approx_{\varphi} {\|f\|}_{\dot{H}^{s}(\mathbb{R}^{n})},
\end{equation}
and ${\|\cdot\|}_{\dot{B}_{p,q}^{s}(\mathbb{R}^{n})}$ is homogeneous of degree $\alpha(\dot{B}_{p,q}^{s}(\mathbb{R}^{n}))=\frac{n}{p}-s$. In the context of Navier-Stokes we define
\begin{equation*}
    s_{p} = s_{p}(n) := -1+\frac{n}{p},
\end{equation*}
so the norm ${\|\cdot\|}_{\dot{B}_{p,q}^{s_{p}+\epsilon}(\mathbb{R}^{n})}$ is critical for $\epsilon=0$, and subcritical for $\epsilon>0$.

It is known that if a regular solution $u$ blows up at a finite time $T>0$, then\footnote{These estimates follow from the local theory and regularity properties of mild solutions with initial data $f$, where the existence time is bounded below by a constant multiple of ${\|f\|}_{\dot{B}_{\infty,\infty}^{-1+\epsilon}(\mathbb{R}^{n})}^{-2/\epsilon}\vee{\|f\|}_{L^{\infty}(\mathbb{R}^{n})}^{-2}$. Estimate \eqref{intro-leray} comes from Leray \cite{leray1934}, while the local theory for initial data in Besov spaces is discussed in Lemari\'{e}-Rieusset's book \cite{lemarie2016}. We aim to give a more detailed account of the regularity properties of such solutions in an upcoming paper.}
\begin{equation}\label{intro-easy-besov}
    {\|u(t)\|}_{\dot{B}_{\infty,\infty}^{-1+\epsilon}(\mathbb{R}^{n})} \gtrsim_{\varphi,\epsilon} {(T-t)}^{-\epsilon/2} \quad \text{for all }t\in(0,T),\,\epsilon\in(0,1),
\end{equation}
\begin{equation}\label{intro-leray}
    {\|u(t)\|}_{L^{\infty}(\mathbb{R}^{n})} \gtrsim_{n} {(T-t)}^{-1/2} \quad \text{for all }t\in(0,T).
\end{equation}
By virtue of \eqref{intro-besov-embedding} and \eqref{intro-lebesgue}, the left-hand side of \eqref{intro-easy-besov} may be replaced by ${\|u(t)\|}_{\dot{B}_{p,q}^{s_{p}+\epsilon}(\mathbb{R}^{n})}$ for any $p,q\in[1,\infty]$, while the left-hand side of \eqref{intro-leray} may be replaced by ${\|u(t)\|}_{\dot{B}_{p,1}^{s_{p}+1}(\mathbb{R}^{n})}$ for any $p\in[1,\infty]$.

Adapting the energy methods of \cite{mccormick2016, robinson2014, robinson2012}, we will prove the following blowup estimates in the case $\epsilon\in[1,2]$:
\begin{theorem}\label{main-theorem}
    Let $n\geq3$ and $T\in(0,\infty)$. If $u$ is a regular solution (see Definition \ref{regsoldef}) to the standardised Navier-Stokes equations on $(0,T)$  which
    satisfies\footnote{Note that this is a natural assumption for a solution blowing up at a finite time $T$ in view of \eqref{intro-easy-besov} with $\epsilon =\tfrac 12$. In fact, to prove the parts of \eqref{eps-less-2} - \eqref{eps-is-2} with $\epsilon \neq 1$, one may replace this assumption with $\lim_{t\nearrow T}{\|u(t)\|}_{L^\infty(\mathbb{R}^{n})}=\infty$ (coming from \eqref{intro-leray}) by replacing \eqref{interpbesonehalf} in the proof with the  estimate
    $$    {\|u(t)\|}_{L^\infty(\mathbb{R}^n)} \leq {\|u(t)\|}_{\dot{B}_{\infty,1}^{0}(\mathbb{R}^n)} \lesssim_{n,\epsilon}{\|u(t)\|}_{\dot{B}_{\infty,\infty}^{-n/2}(\mathbb{R}^n)}^{\lambda}{\|u(t)\|}_{\dot{B}_{\infty,\infty}^{-1+\epsilon}(\mathbb{R}^n)}^{1-\lambda} \quad \text{with} \quad \lambda=\frac{\epsilon-1}{\epsilon-1+\frac{n}{2}}\, .
    $$
    }
      $\lim_{t\nearrow T}{\|u(t)\|}_{\dot{B}_{\infty,\infty}^{-1/2}(\mathbb{R}^{n})}=\infty$, then
    \begin{equation}\label{eps-less-2}
        {\|u(t)\|}_{\dot{B}_{p,q}^{s_{p}+\epsilon}(\mathbb{R}^{n})} \gtrsim_{\varphi,\epsilon,(p\vee q\vee2)} {(T-t)}^{-\epsilon/2} \quad \text{for all }t\in(0,T),\,\epsilon\in[1,2),\,p,q\in\left[1,\frac{n}{2-\epsilon}\right)
    \end{equation}
and
    \begin{equation}\label{eps-is-2}
        {\|u(t)\|}_{\dot{B}_{p,1}^{s_{p}+2}(\mathbb{R}^{n})} \gtrsim_{\varphi,(p\vee2)} {(T-t)}^{-1} \quad \text{for all }t\in(0,T),\,p\in[1,\infty).
    \end{equation}
\end{theorem}
Under the additional restrictions that $p,q\in[1,2]$ and $n=3$, the blowup estimate \eqref{eps-less-2} is implied by the blowup estimate for $\dot{H}^{s_{2}+\epsilon}(\mathbb{R}^{3})$, which was proved in the case $\epsilon\in(1,2)$ by Robinson, Sadowski and Silva \cite{robinson2012}, and in the case $\epsilon=1$ by McCormick et al.\ \cite{mccormick2016}. Under the additional restrictions that $p\in[1,2]$ and $n=3$, the blowup estimate \eqref{eps-is-2} is implied by the blowup estimate for $\dot{B}_{2,1}^{5/2}(\mathbb{R}^{3})$, which was proved by McCormick et al.\ \cite{mccormick2016}.

The rest of this paper is organised as follows. In section \ref{besov-spaces} we recall some standard properties of Besov spaces, using \cite{bahouri2011} as our main reference. In section \ref{commutator-estimates} we prove some commutator estimates, adapting the ideas of \cite[Lemma 2.100]{bahouri2011}. In section \ref{navier-stokes-blowup-rates} we prove Theorem \ref{main-theorem}. We will henceforth use the abbreviations $L^{p}=L^{p}(\mathbb{R}^{n})$, $\dot{H}^{s}=\dot{H}^{s}(\mathbb{R}^{n})$, $\dot{B}_{p,q}^{s}=\dot{B}_{p,q}^{s}(\mathbb{R}^{n})$ and $l^{q}=l^{q}(\mathbb{Z})$.
\section{Besov spaces}\label{besov-spaces}
\begin{lemma}\label{bahouri-prop-2.10}
    (\cite{bahouri2011}, Proposition 2.10). Let $\mathcal{C}$ be the annulus $B(0,8/3)\setminus\overline{B}(0,3/4)$. Then the set $\widetilde{\mathcal{C}}=B(0,2/3)+\mathcal{C}$ is an annulus, and there exist radial functions $\chi\in\mathcal{D}(B(0,4/3))$ and $\varphi\in\mathcal{D}(\mathcal{C})$, taking values in $[0,1]$, such that
    \begin{equation*}
        \left\{\begin{array}{ll}
            \chi(\xi)+\sum_{j\geq0}\varphi(2^{-j}\xi)=1 & \forall\,\xi\in\mathbb{R}^{n}, \\
            \sum_{j\in\mathbb{Z}}\varphi(2^{-j}\xi)=1 & \forall\,\xi\in\mathbb{R}^{n}\setminus\{0\}, \\
            |j-j'|\geq2\Rightarrow\supp\varphi(2^{-j}\cdot)\cap\supp\varphi(2^{-j'}\cdot)=\emptyset, \\
            j\geq1\Rightarrow\supp\chi\cap\supp\varphi(2^{-j}\cdot)=\emptyset, \\
            |j-j'|\geq5\Rightarrow2^{j'}\widetilde{\mathcal{C}}\cap2^{j}\mathcal{C}=\emptyset, \\
            1/2\leq\chi^{2}(\xi)+\sum_{j\geq0}\varphi^{2}(2^{-j}\xi)\leq1 & \forall\,\xi\in\mathbb{R}^{n}, \\
            1/2\leq\sum_{j\in\mathbb{Z}}\varphi^{2}(2^{-j}\xi)\leq1 & \forall\,\xi\in\mathbb{R}^{n}\setminus\{0\}. \\
        \end{array}\right.
    \end{equation*}
\end{lemma}
We fix $\chi,\varphi$ satisfying Lemma \ref{bahouri-prop-2.10}. For $j\in\mathbb{Z}$ and $u\in\mathcal{S}'$, we define\footnote{We adopt the convention that $\mathcal{F}f(\xi) = \int_{\mathbb{R}^{n}}e^{-\mathrm{i}\xi\cdot x}f(x)\,\mathrm{d}x$ for $f\in\mathcal{S}$. We recall that the Fourier transform of a compactly supported distribution is a smooth function.}
\begin{equation*}
    \dot{S}_{j}u := \chi(2^{-j}D)u = \mathcal{F}^{-1}\chi(2^{-j}\xi)\mathcal{F}u,
\end{equation*}
\begin{equation*}
    \dot{\Delta}_{j}u := \varphi(2^{-j}D)u = \mathcal{F}^{-1}\varphi(2^{-j}\xi)\mathcal{F}u.
\end{equation*}
\begin{lemma}\label{truncation-lemma}
    For any $j,j'\in\mathbb{Z}$ and $u,v\in\mathcal{S}'$, we have
    \begin{equation*}
        |j-j'|\geq 2\Rightarrow \dot{\Delta}_{j}\dot{\Delta}_{j'}u=0, \qquad |j-j'|\geq5\Rightarrow\dot{\Delta}_{j}\left(\dot{S}_{j'-1}u\,\dot{\Delta}_{j'}v\right)=0.
    \end{equation*}
\end{lemma}
\begin{proof}
    This is a consequence of Lemma \ref{bahouri-prop-2.10}. In particular, the implication $|j-j'|\geq 2\Rightarrow \dot{\Delta}_{j}\dot{\Delta}_{j'}u=0$ follows from the implication $|j-j'|\geq2\Rightarrow\supp\varphi(2^{-j}\cdot)\cap\supp\varphi(2^{-j'}\cdot)=\emptyset$, while the implication $|j-j'|\geq5\Rightarrow\dot{\Delta}_{j}\left(\dot{S}_{j'-1}u\,\dot{\Delta}_{j'}v\right)=0$ follows\footnote{Note that $\dot{S}_{j'-1}u$ is spectrally supported on $2^{j'}B(0,2/3)$, while $\dot{\Delta}_{j'}v$ is spectrally supported on $2^{j'}\mathcal{C}$, so by properties of convolution we have that $\dot{S}_{j'-1}u\,\dot{\Delta}_{j'}v$ is spectrally supported on $2^{j'}\widetilde{\mathcal{C}}$.} from the implication $|j-j'|\geq5\Rightarrow2^{j'}\widetilde{\mathcal{C}}\cap2^{j}\mathcal{C}=\emptyset$.
\end{proof}

We recall the following useful properties:
\begin{lemma}\label{useful-inequalities}
    (\cite{bahouri2011}, Lemmas 2.1-2.2, Remark 2.11). Let $\rho$ be a smooth function on $\mathbb{R}^{n}\setminus\{0\}$ which is positive homogeneous of degree $\lambda\in\mathbb{R}$. Then for all $j\in\mathbb{Z}$, $u\in\mathcal{S}'$, $t\in(0,\infty)$ and $1\leq p\leq q\leq\infty$ we have
    \begin{equation}\label{useful-inequality-1}
        {\|\dot{S}_{j}u\|}_{L^{p}}\vee{\|\dot{\Delta}_{j}u\|}_{L^{p}} \lesssim_{\varphi} {\|u\|}_{L^{p}},
    \end{equation}
    \begin{equation}
        {\|\rho(D)\dot{\Delta}_{j}u\|}_{L^{q}} \lesssim_{\rho} 2^{j\lambda}2^{j\left(\frac{n}{p}-\frac{n}{q}\right)}{\|\dot{\Delta}_{j}u\|}_{L^{p}},
    \end{equation}
    \begin{equation}
        {\|\dot{\Delta}_{j}u\|}_{L^{p}} \lesssim_{n} 2^{-j}{\|\nabla\dot{\Delta}_{j}u\|}_{L^{p}}.
    \end{equation}
\end{lemma}
One can give meaning to the decomposition $u=\sum_{j\in\mathbb{Z}}\dot{\Delta}_{j}u$ in view of the following lemma:
\begin{lemma}\label{littlewood-paley-decomposition}
    (\cite{bahouri2011}, Propositions 2.12-2.14) If $u\in\mathcal{S}'$, then $\dot{S}_{j}u\overset{j\rightarrow\infty}{\rightarrow}u$ in $\mathcal{S}'$. Define\footnote{For example, if $\mathcal{F}u$ is locally integrable near $\xi=0$, then $u\in\mathcal{S}_{h}'$. We remark that the condition $u\in\mathcal{S}_{h}'$ is independent of our choice of $\varphi$.}
    \begin{equation*}
        \mathcal{S}_{h}' := \left\{u\in\mathcal{S}'\text{ }:\text{ }{\|\dot{S}_{j}u\|}_{L^{\infty}}\overset{j\rightarrow-\infty}{\rightarrow}0\right\},
    \end{equation*}
    so if $u\in\mathcal{S}_{h}'$ then $u=\sum_{j\in\mathbb{Z}}\dot{\Delta}_{j}u$ in $\mathcal{S}'$.
\end{lemma}
For $s\in\mathbb{R}$ and $p,q\in[1,\infty]$, we define the Besov seminorm\footnote{Choosing a different function $\varphi$ yields an equivalent seminorm \cite[Remark 2.17]{bahouri2011}.}
\begin{equation*}
    {\|u\|}_{\dot{B}_{p,q}^{s}} := {\left\|j\mapsto2^{js}{\|\dot{\Delta}_{j}u\|}_{L^{p}}\right\|}_{l^{q}} \quad \text{for }u\in\mathcal{S}'
\end{equation*}
and the Besov space
\begin{equation*}
    \dot{B}_{p,q}^{s} := \left\{u\in\mathcal{S}_{h}'\text{ : }{\|u\|}_{\dot{B}_{p,q}^{s}}<\infty\right\},
\end{equation*}
so that $\left(\dot{B}_{p,q}^{s},{\|\cdot\|}_{\dot{B}_{p,q}^{s}}\right)$ is a normed space \cite[Proposition 2.16]{bahouri2011}. Lemma \ref{useful-inequalities} and Lemma \ref{littlewood-paley-decomposition} yield the inequalities
\begin{equation}\label{besov-embedding}
    {\|u\|}_{\dot{B}_{p_{2},q_{2}}^{\frac{n}{p_{2}}+\epsilon}} \lesssim_{n} {\|u\|}_{\dot{B}_{p_{1},q_{1}}^{\frac{n}{p_{1}}+\epsilon}} \quad \text{for }p_{1}\leq p_{2},\,q_{1}\leq\,q_{2},\,\epsilon\in\mathbb{R},\,u\in\mathcal{S}',
\end{equation}
\begin{equation}\label{rough-lp}
    {\|u\|}_{\dot{B}_{p,\infty}^{0}} \lesssim_{\varphi} {\|u\|}_{L^{p}} \quad \text{for }u\in\mathcal{S}'
\end{equation}
and
\begin{equation}\label{smooth-lp}
    {\|u\|}_{L^{p}} \leq {\|u\|}_{\dot{B}_{p,1}^{0}} \quad \text{for }u\in\mathcal{S}_{h}'.
\end{equation}
We also have the interpolation inequalities
\begin{equation}\label{interpolation-holder}
    {\|u\|}_{\dot{B}_{\frac{p_{1}p_{2}}{\lambda p_{2}+(1-\lambda)p_{1}},\frac{q_{1}q_{2}}{\lambda q_{2}+(1-\lambda)q_{1}}}^{\lambda s_{1}+(1-\lambda)s_{2}}} \leq {\|u\|}_{\dot{B}_{p_{1},q_{1}}^{s_{1}}}^{\lambda}{\|u\|}_{\dot{B}_{p_{2},q_{2}}^{s_{2}}}^{1-\lambda} \quad \text{for }\lambda\in(0,1),\,u\in\mathcal{S}',
\end{equation}
\begin{equation}\label{interpolation-geometric}
    {\|u\|}_{\dot{B}_{p,1}^{\lambda s_{1}+(1-\lambda)s_{2}}} \lesssim \frac{1}{\lambda(1-\lambda)(s_{2}-s_{1})}{\|u\|}_{\dot{B}_{p,\infty}^{s_{1}}}^{\lambda}{\|u\|}_{\dot{B}_{p,\infty}^{s_{2}}}^{1-\lambda} \quad \text{for }\lambda\in(0,1),\,s_{1}<s_{2},\,u\in\mathcal{S}',
\end{equation}
where \eqref{interpolation-holder} comes from H\"{o}lder's inequality, while \eqref{interpolation-geometric} comes from writing $\sum_{j\in\mathbb{Z}}=\sum_{j\leq j_{0}}+\sum_{j>j_{0}}$ with $2^{j_{0}(s_{2}-s_{1})}{\|u\|}_{\dot{B}_{p,\infty}^{s_{1}}}\approx{\|u\|}_{\dot{B}_{p,\infty}^{s_{2}}}$ and applying geometric series.

We now recall the following convergence lemma:
\begin{lemma}\label{convergence-lemma}
    (\cite{bahouri2011}, Lemma 2.23). Let $\mathcal{C}'$ be an annulus and ${(u_{j})}_{j\in\mathbb{Z}}$ be a sequence of functions such that $\supp\mathcal{F}u_{j}\subseteq2^{j}\mathcal{C}'$ and ${\left\|j\mapsto2^{js}{\|u_{j}\|}_{L^{p}}\right\|}_{l^{q}}<\infty$. If the series $\sum_{j\in\mathbb{Z}}u_{j}$ converges in $\mathcal{S}'$ to some $u\in\mathcal{S}'$, then
    \begin{equation*}\label{convergence-inequality}
        {\|u\|}_{\dot{B}_{p,q}^{s}} \lesssim_{\varphi} C_{\mathcal{C}'}^{1+|s|}{\left\|j\mapsto2^{js}{\|u_{j}\|}_{L^{p}}\right\|}_{l^{q}}.
    \end{equation*}
    Note: If $(s,p,q)$ satisfy the condition
    \begin{equation}\label{negative-scaling}
        s<\frac{n}{p}, \quad \text{or} \quad s=\frac{n}{p}\text{ and }q=1,
    \end{equation}
    then the hypothesis of convergence is satisfied, and $u\in\mathcal{S}_{h}'$.
\end{lemma}
A useful consequence of Lemma \ref{convergence-lemma} is that if $u\in\mathcal{S}'$ satisfies ${\|u\|}_{\dot{B}_{p,q}^{s}}<\infty$ for some $(s,p,q)$ satisfying \eqref{negative-scaling}, then $u\in\mathcal{S}_{h}'$.

If $u\in\dot{B}_{p_{1},1}^{0}$ and $v\in\dot{B}_{p_{2},1}^{0}$ with $\frac{1}{p_{1}}+\frac{1}{p_{2}}\leq1$, then the series $uv=\sum_{(j,j')\in\mathbb{Z}^{2}}\dot{\Delta}_{j}u\,\dot{\Delta}_{j'}v$ converges absolutely in $L^{\frac{p_{1}p_{2}}{p_{1}+p_{2}}}$, which justifies the Bony decomposition
\begin{equation*}
    uv = \dot{T}_{u}v+\dot{T}_{v}u+\dot{R}(u,v),
\end{equation*}
\begin{equation*}
    \dot{T}_{u}v = \sum_{j\in\mathbb{Z}}\dot{S}_{j-1}u\,\dot{\Delta}_{j}v,
\end{equation*}
\begin{equation*}
    \dot{R}(u,v) = \sum_{j\in\mathbb{Z}}\sum_{|\nu|\leq1}\dot{\Delta}_{j}u\,\dot{\Delta}_{j-\nu}v.
\end{equation*}
We will require the following estimates for the operators $\dot{T}$ and $\dot{R}$:
\begin{lemma}\label{paraproduct}
    (\cite{bahouri2011}, Theorem 2.47). Suppose that $s=s_{1}+s_{2}$, $p=\frac{p_{1}p_{2}}{p_{1}+p_{2}}$ and $q=\frac{q_{1}q_{2}}{q_{1}+q_{2}}$. Let $u,v\in\mathcal{S}'$, and assume that the series $\sum_{j\in\mathbb{Z}}\dot{S}_{j-1}u\,\dot{\Delta}_{j}v$ converges in $\mathcal{S}'$ to some $\dot{T}_{u}v\in\mathcal{S}'$. Then
    \begin{equation}\label{bony-estimate-1}
        {\|\dot{T}_{u}v\|}_{\dot{B}_{p,q}^{s}} \lesssim_{\varphi} C_{n}^{1+|s|}{\|u\|}_{L^{p_{1}}}{\|v\|}_{\dot{B}_{p_{2},q}^{s}},
    \end{equation}
    \begin{equation}\label{bony-estimate-2}
        {\|\dot{T}_{u}v\|}_{\dot{B}_{p,q}^{s}} \lesssim_{\varphi} \frac{C_{n}^{1+|s|}}{-s_{1}}{\|u\|}_{\dot{B}_{p_{1},q_{1}}^{s_{1}}}{\|v\|}_{\dot{B}_{p_{2},q_{2}}^{s_{2}}} \quad \text{if } s_{1}<0.
    \end{equation}
    Note: If $(s,p,q)$ satisfy \eqref{negative-scaling}, and the right hand side of either \eqref{bony-estimate-1} or \eqref{bony-estimate-2} is finite, then the hypothesis of convergence is satisfied, and $\dot{T}_{u}v\in\mathcal{S}_{h}'$.
\end{lemma}
\begin{lemma}\label{remainder}
    (\cite{bahouri2011}, Theorem 2.52). Suppose that $s=s_{1}+s_{2}$, $p=\frac{p_{1}p_{2}}{p_{1}+p_{2}}$ and $q=\frac{q_{1}q_{2}}{q_{1}+q_{2}}$. Let $u,v\in\mathcal{S}'$, and assume that the series $\sum_{j\in\mathbb{Z}}\sum_{|\nu|\leq1}\dot{\Delta}_{j}u\,\dot{\Delta}_{j-\nu}v$ converges in $\mathcal{S}'$ to some $\dot{R}(u,v)\in\mathcal{S}'$. Then
    \begin{equation}\label{bony-estimate-3}
        {\|\dot{R}(u,v)\|}_{\dot{B}_{p,q}^{s}} \lesssim_{\varphi} \frac{C_{n}^{1+|s|}}{s}{\|u\|}_{\dot{B}_{p_{1},q_{1}}^{s_{1}}}{\|v\|}_{\dot{B}_{p_{2},q_{2}}^{s_{2}}} \quad \text{if } s>0,
    \end{equation}
    \begin{equation}\label{bony-estimate-4}
        {\|\dot{R}(u,v)\|}_{\dot{B}_{p,\infty}^{s}} \lesssim_{\varphi} C_{n}^{1+|s|}{\|u\|}_{\dot{B}_{p_{1},q_{1}}^{s_{1}}}{\|v\|}_{\dot{B}_{p_{2},q_{2}}^{s_{2}}} \quad \text{if }q=1\text{ and } s\geq0.
    \end{equation}
    Note: If $(s,p,q)$ satisfy \eqref{negative-scaling} and the right hand side of \eqref{bony-estimate-3} is finite, or if $(s,p,\infty)$ satisfy \eqref{negative-scaling} and the right hand side of \eqref{bony-estimate-4} is finite, then the hypothesis of convergence is satisfied, and $\dot{R}(u,v)\in\mathcal{S}_{h}'$.
\end{lemma}
\section{Commutator estimates}\label{commutator-estimates}
In this section, we will adapt the proof of \cite[Lemma 2.100]{bahouri2011} to prove the commutator estimates in the following proposition, which will be crucial to the proof of Theorem \ref{main-theorem}:
\begin{proposition}\label{commutator-prop}
    For $v,f\in\cap_{r\in[0,\infty)}\dot{H}^{r}$ and $j\in\mathbb{Z}$, define\footnote{We apply the summation convention to the index $k$.}
    \begin{equation*}
        R_{j}=[v\cdot\nabla,\dot{\Delta}_{j}]f=[v_{k},\dot{\Delta}_{j}]\nabla_{k}f
    \end{equation*}
    where $[\,\cdot\,,\,\cdot\,]$ denotes the commutator $[A,B]=AB-BA$, and suppose that $s=s_{1}+s_{2}$, $p=\frac{p_{1}p_{2}}{p_{1}+p_{2}}$ and $q=\frac{q_{1}q_{2}}{q_{1}+q_{2}}$ ($s_j\in \mathbb{R}$ and $p_j,q_j \in [1,\infty]$).
    Then we have the decomposition $R_{j}=\sum_{i=1}^{6}R_{j}^{i}$ with\footnote{Note that $R_{j}^{6}=0$ whenever $\nabla\cdot v=0$.}
    \begin{equation*}
        R_{j}^{1} = [\dot{T}_{v_{k}},\dot{\Delta}_{j}]\nabla_{k}f, \quad R_{j}^{2} = \dot{T}_{\nabla_{k}\dot{\Delta}_{j}f}v_{k}, \quad R_{j}^{3} = -\dot{\Delta}_{j}\dot{T}_{\nabla_{k}f}v_{k},
    \end{equation*}
    \begin{equation*}
        R_{j}^{4} = \dot{R}(v_{k},\nabla_{k}\dot{\Delta}_{j}f), \quad R_{j}^{5} = -\nabla_{k}\dot{\Delta}_{j}\dot{R}(v_{k},f), \quad R_{j}^{6} = \dot{\Delta}_{j}\dot{R}(\nabla_{k}v_{k},f)
    \end{equation*}
  which satisfy the estimates
    \begin{equation*}
    \begin{aligned}
        {\left\|j\mapsto2^{js}{\|R_{j}^{1}\|}_{L^{p}}\right\|}_{l^{q}} &\lesssim {\|\nabla v\|}_{L^{p_{1}}}{\|f\|}_{\dot{B}_{p_{2},q}^{s}}, \\
        {\left\|j\mapsto2^{js}{\|R_{j}^{1}\|}_{L^{p}}\right\|}_{l^{q}} &\lesssim {\|\nabla v\|}_{\dot{B}_{p_{1},q_{1}}^{s_{1}}}{\|f\|}_{\dot{B}_{p_{2},q_{2}}^{s_{2}}} \quad \text{if }s_{1}<0, \\
        {\left\|j\mapsto2^{js}{\|R_{j}^{2}\|}_{L^{p}}\right\|}_{l^{q}} &\lesssim {\|\nabla v\|}_{\dot{B}_{p_{1},q_{1}}^{s_{1}}}{\|f\|}_{\dot{B}_{p_{2},q_{2}}^{s_{2}}} \quad \text{if }s_{1}>-1, \\
        {\left\|j\mapsto2^{js}{\|R_{j}^{3}\|}_{L^{p}}\right\|}_{l^{q}} &\lesssim {\|\nabla v\|}_{\dot{B}_{p_{1},q_{1}}^{s_{1}}}{\|f\|}_{\dot{B}_{p_{2},q_{2}}^{s_{2}}} \quad \text{if }s_{2}<1, \\
        {\left\|j\mapsto2^{js}{\|R_{j}^{4}\|}_{L^{p}}\right\|}_{l^{q}} &\lesssim {\|\nabla v\|}_{\dot{B}_{p_{1},q_{1}}^{s_{1}}}{\|f\|}_{\dot{B}_{p_{2},q_{2}}^{s_{2}}}, \\
        {\left\|j\mapsto2^{js}{\|R_{j}^{5}\|}_{L^{p}}\right\|}_{l^{q}} &\lesssim {\|\nabla v\|}_{\dot{B}_{p_{1},q_{1}}^{s_{1}}}{\|f\|}_{\dot{B}_{p_{2},q_{2}}^{s_{2}}} \quad \text{if }s>-1, \\
        {\left\|j\mapsto2^{js}{\|R_{j}^{6}\|}_{L^{p}}\right\|}_{l^{q}} &\lesssim {\|\nabla v\|}_{\dot{B}_{p_{1},q_{1}}^{s_{1}}}{\|f\|}_{\dot{B}_{p_{2},q_{2}}^{s_{2}}} \quad \text{if }s>0,
    \end{aligned}
    \end{equation*}
    where the implied constants depend on $\varphi,s_{1},s_{2},p_{1},p_{2}$.
\end{proposition}
\begin{remark}
    Write $A_{1},A_{2},A_{3},A_{5},A_{6}$ to denote the constraints
    \begin{equation*}
        A_{1}\text{ : }s_{1}\leq0, \quad A_{2}\text{ : }s_{1}\geq-1, \quad A_{3}\text{ : }s_{2}\leq1, \quad A_{5}\text{ : }s\geq-1, \quad A_{6}\text{ : }s\geq0.
    \end{equation*}
    For $i=1,2,3,5,6$, a simple modification of our arguments yields the estimates
    \begin{equation*}
        {\left\|j\mapsto2^{js}{\|R_{j}^{i}\|}_{L^{p}}\right\|}_{l^{\infty}} \lesssim {\|\nabla v\|}_{\dot{B}_{p_{1},q_{1}}^{s_{1}}}{\|f\|}_{\dot{B}_{p_{2},q_{2}}^{s_{2}}} \quad \text{if }q=1\text{ and }A_{i}\text{ holds},
    \end{equation*}
    but we will not need these estimates when proving Theorem \ref{main-theorem}.
\end{remark}
As in \cite[Lemma 2.100]{bahouri2011}, to prove Proposition \ref{commutator-prop} we will rely on the following lemma:

\begin{lemma}\label{commutator-lemma}
    (\cite{bahouri2011}, Lemma 2.97). Let $\theta\in C^{1}(\mathbb{R}^{n})$ be such that $\int_{\mathbb{R}^{n}}(1+|\xi|)|\mathcal{F}\theta(\xi)|\,\mathrm{d}\xi<\infty.$ Then for any $a\in C^{1}(\mathbb{R}^{n})$ with $\nabla a\in L^{p}(\mathbb{R}^{n})$, any $b\in L^{q}(\mathbb{R}^{n})$, and any $\lambda\in(0,\infty)$, we have
    \begin{equation*}
        {\|[\theta(\lambda^{-1}D),a]b\|}_{L^{\frac{pq}{p+q}}(\mathbb{R}^{n})} \lesssim_{\theta} \lambda^{-1}{\|\nabla a\|}_{L^{p}(\mathbb{R}^{n})}{\|b\|}_{L^{q}(\mathbb{R}^{n})}.
    \end{equation*}
\end{lemma}
\begin{remark}\label{commutator-remark}
    If we take $\theta=\varphi$ and $\lambda=2^{j}$, then Lemma \ref{commutator-lemma} yields the estimate
    \begin{equation*}
        {\|[\dot{\Delta}_{j},a]b\|}_{L^{\frac{pq}{p+q}}(\mathbb{R}^{n})} \lesssim_{\varphi} 2^{-j}{\|\nabla a\|}_{L^{p}(\mathbb{R}^{n})}{\|b\|}_{L^{q}(\mathbb{R}^{n})}.
    \end{equation*}
\end{remark}
\begin{proof}[Proof of Proposition \ref{commutator-prop}]
    The decomposition $R_{j}=\sum_{i=1}^{6}R_{j}^{i}$ comes from applying the Bony decomposition; the very strong regularity assumption $v,f\in\cap_{r\in[0,\infty)}\dot{H}^{r}$ is more than sufficient to address any convergence issues that may arise. In the following computations, we write ${(c_{j})}_{j\in\mathbb{Z}}$ to denote a sequence satisfying ${\|(c_{j})\|}_{l^{q}}\leq1$, and the constants implied by the notation $\lesssim$ depend on $\varphi,s_{1},s_{2},p_{1},p_{2}$.

    {\bf Bounds for $\boldsymbol{2^{js}{\|R_{j}^{1}\|}_{L^{p}}}$.} By Lemma \ref{truncation-lemma} we have
    \begin{equation*}
        R_{j}^{1} = \sum_{|j-j'|\leq4}[\dot{S}_{j'-1}v_{k},\dot{\Delta}_{j}]\nabla_{k}\dot{\Delta}_{j'}f,
    \end{equation*}
    so by Remark \ref{commutator-remark} and Lemma \ref{useful-inequalities} we have
    \begin{equation}\label{R1-estimate}
        2^{js}{\|R_{j}^{1}\|}_{L^{p}} \lesssim \sum_{|j-j'|\leq4}2^{js}2^{j'-j}{\|\nabla\dot{S}_{j'-1}v\|}_{L^{p_{1}}}{\|\dot{\Delta}_{j'}f\|}_{L^{p_{2}}}.
    \end{equation}
    By \eqref{useful-inequality-1}, we deduce that
    \begin{equation*}
        2^{js}{\|R_{j}^{1}\|}_{L^{p}} \lesssim \sum_{|j-j'|\leq4}2^{js}2^{j'-j}{\|\nabla v\|}_{L^{p_{1}}}{\|\dot{\Delta}_{j'}f\|}_{L^{p_{2}}} \lesssim c_{j}{\|\nabla v\|}_{L^{p_{1}}}{\|f\|}_{\dot{B}_{p_{2},q}^{s}}.
    \end{equation*}
    On the other hand, if $s_{1}<0$ then \eqref{R1-estimate} implies that
    \begin{equation*}
    \begin{aligned}
        2^{js}{\|R_{j}^{1}\|}_{L^{p}} &\lesssim \sum_{\substack{|j-j'|\leq4 \\ j''\leq j'-2}}2^{js}2^{j'-j}{\|\nabla\dot{\Delta}_{j''}v\|}_{L^{p_{1}}}{\|\dot{\Delta}_{j'}f\|}_{L^{p_{2}}} \\
        &\lesssim \sum_{\substack{|j-j'|\leq4 \\ j''\leq j'-2}} 2^{(j-j'')s_{1}}2^{j''s_{1}}{\|\nabla\dot{\Delta}_{j''}v\|}_{L^{p_{1}}}2^{j's_{2}}{\|\dot{\Delta}_{j'}f\|}_{L^{p_{2}}} \\
        &\lesssim c_{j}{\|\nabla v\|}_{\dot{B}_{p_{1},q_{1}}^{s_{1}}}{\|f\|}_{\dot{B}_{p_{2},q_{2}}^{s_{2}}},
    \end{aligned}
    \end{equation*}
    where we used the inequality ${\|(\alpha*\beta)\gamma\|}_{l^{q}}\leq{\|\alpha*\beta\|}_{l^{q_{1}}}{\|\gamma\|}_{l^{q_{2}}}\leq{\|\alpha\|}_{l^{1}}{\|\beta\|}_{l^{q_{1}}}{\|\gamma\|}_{l^{q_{2}}}$ in the last line.

    {\bf Bounds for $\boldsymbol{2^{js}{\|R_{j}^{2}\|}_{L^{p}}}$.} By Lemma \ref{truncation-lemma} we have
    \begin{equation*}
        R_{j}^{2} = \sum_{j'\geq j+1}\dot{S}_{j'-1}\nabla_{k}\dot{\Delta}_{j}f\,\dot{\Delta}_{j'}v_{k}.
    \end{equation*}
    so by Lemma \ref{useful-inequalities} we have
    \begin{equation*}
        2^{js}{\|R_{j}^{2}\|}_{L^{p}} \lesssim \sum_{j'\geq j+1}2^{js}2^{j-j'}{\|\nabla\dot{\Delta}_{j'}v\|}_{L^{p_{1}}}{\|\dot{\Delta}_{j}f\|}_{L^{p_{2}}}.
    \end{equation*}
    If $s_{1}>-1$, then we deduce that
    \begin{equation*}
    \begin{aligned}
        2^{js}{\|R_{j}^{2}\|}_{L^{p}} &\lesssim \sum_{j'\geq j+1}2^{(j-j')(s_{1}+1)}2^{j's_{1}}{\|\nabla\dot{\Delta}_{j'}v\|}_{L^{p_{1}}}2^{js_{2}}{\|\dot{\Delta}_{j}f\|}_{L^{p_{2}}} \\
        &\lesssim c_{j}{\|\nabla v\|}_{\dot{B}_{p_{1},q_{1}}^{s_{1}}}{\|f\|}_{\dot{B}_{p_{2},q_{2}}^{s_{2}}},
    \end{aligned}
    \end{equation*}
    where we used the inequality ${\|(\alpha*\beta)\gamma\|}_{l^{q}}\leq{\|\alpha*\beta\|}_{l^{q_{1}}}{\|\gamma\|}_{l^{q_{2}}}\leq{\|\alpha\|}_{l^{1}}{\|\beta\|}_{l^{q_{1}}}{\|\gamma\|}_{l^{q_{2}}}$ in the last line.

    {\bf Bounds for $\boldsymbol{2^{js}{\|R_{j}^{3}\|}_{L^{p}}}$.} By Lemma \ref{truncation-lemma} we have
    \begin{equation*}
    \begin{aligned}
        R_{j}^{3} &= -\sum_{|j-j'|\leq4}\dot{\Delta}_{j}\left(\dot{S}_{j'-1}\nabla_{k}f\,\dot{\Delta}_{j'}v_{k}\right) \\
        &= -\sum_{\substack{|j-j'|\leq4 \\ j''\leq j'-2}}\dot{\Delta}_{j}\left(\dot{\Delta}_{j''}\nabla_{k}f\,\dot{\Delta}_{j'}v_{k}\right),
    \end{aligned}
    \end{equation*}
    so by Lemma \ref{useful-inequalities} we have
    \begin{equation*}
        2^{js}{\|R_{j}^{3}\|}_{L^{p}} \lesssim \sum_{\substack{|j-j'|\leq4 \\ j''\leq j'-2}}2^{js}2^{j''-j'}{\|\nabla\dot{\Delta}_{j'}v\|}_{L^{p_{1}}}{\|\dot{\Delta}_{j''}f\|}_{L^{p_{2}}}.
    \end{equation*}
    If $s_{2}<1$, then we deduce that
    \begin{equation*}
    \begin{aligned}
        2^{js}{\|R_{j}^{3}\|}_{L^{p}} &\lesssim \sum_{\substack{|j-j'|\leq4 \\ j''\leq j'-2}}2^{j's_{1}}{\|\nabla\dot{\Delta}_{j'}v\|}_{L^{p_{1}}}2^{(j'-j'')(s_{2}-1)}2^{j''s_{2}}{\|\dot{\Delta}_{j''}f\|}_{L^{p_{2}}} \\
        &\lesssim c_{j}{\|\nabla v\|}_{\dot{B}_{p_{1},q_{1}}^{s_{1}}}{\|f\|}_{\dot{B}_{p_{2},q_{2}}^{s_{2}}},
    \end{aligned}
    \end{equation*}
    where we used the inequality ${\|\alpha(\beta*\gamma)\|}_{l^{q}}\leq{\|\alpha\|}_{l^{q_{1}}}{\|\beta*\gamma\|}_{l^{q_{2}}}\leq{\|\alpha\|}_{l^{q_{1}}}{\|\beta\|}_{l^{1}}{\|\gamma\|}_{l^{q_{2}}}$ in the last line.

    {\bf Bounds for $\boldsymbol{2^{js}{\|R_{j}^{4}\|}_{L^{p}}}$.} Defining $\widetilde{\Delta}_{j'}=\sum_{|\nu|\leq1}\dot{\Delta}_{j'-\nu}$, by Lemma \ref{truncation-lemma} we have
    \begin{equation*}
        R_{j}^{4} = \sum_{|j-j'|\leq2}\dot{\Delta}_{j'}v_{k}\,\nabla_{k}\dot{\Delta}_{j}\widetilde{\Delta}_{j'}f,
    \end{equation*}
    so by Lemma \ref{useful-inequalities} and the inequality ${\|\alpha\beta\|}_{l^{q}}\leq{\|\alpha\|}_{l^{q_{1}}}{\|\beta\|}_{l^{q_{2}}}$ we have
    \begin{equation*}
        2^{js}{\|R_{j}^{4}\|}_{L^{p}} \lesssim c_{j}{\|\nabla v\|}_{\dot{B}_{p_{1},q_{1}}^{s_{1}}}{\|f\|}_{\dot{B}_{p_{2},q_{2}}^{s_{2}}}.
    \end{equation*}

    {\bf Bounds for $\boldsymbol{2^{js}{\|R_{j}^{5}\|}_{L^{p}}}$ and $\boldsymbol{2^{js}{\|R_{j}^{6}\|}_{L^{p}}}$.} By Lemma \ref{useful-inequalities} and \eqref{bony-estimate-3}, we have
    \begin{equation*}
    \begin{aligned}
        {\left\|j\mapsto2^{js}{\|R_{j}^{5}\|}_{L^{p}}\right\|}_{l^{q}} &\lesssim {\|\nabla v\|}_{\dot{B}_{p_{1},q_{1}}^{s_{1}}}{\|f\|}_{\dot{B}_{p_{2},q_{2}}^{s_{2}}} \quad \text{if }s>-1, \\
        {\left\|j\mapsto2^{js}{\|R_{j}^{6}\|}_{L^{p}}\right\|}_{l^{q}} &\lesssim {\|\nabla v\|}_{\dot{B}_{p_{1},q_{1}}^{s_{1}}}{\|f\|}_{\dot{B}_{p_{2},q_{2}}^{s_{2}}} \quad \text{if }s>0.
    \end{aligned}
    \end{equation*}
\end{proof}
\section{Proof of blowup rates}\label{navier-stokes-blowup-rates}
We now give the
\begin{proof}[Proof of Theorem \ref{main-theorem}]%%%
Note first that the regularity assumptions on $u$ are strong enough to justify the calculations used in this proof.
\\\\
Fix $\epsilon\in[1,2]$ and $p,q\in[1,\frac{n}{2-\epsilon})$.
By virtue of the inequality \eqref{besov-embedding}, it suffices to prove the estimate
\begin{equation}\label{u-blowup}
    {\|u(t)\|}_{\dot{B}_{r,\widetilde{r}}^{s_{r}+\epsilon}} \gtrsim_{\varphi,\epsilon,r} {(T-t)}^{-\epsilon/2}
\end{equation}
for any fixed $r\in[p\vee q\vee2,\frac{n}{2-\epsilon})$ (e.g., $r:=p\vee q\vee2$), where $\widetilde{r}=r$ in the case $\epsilon<2$, and $\widetilde{r}=1$ in the case $\epsilon=2$. By the embeddings $L^{2}\hookrightarrow\dot{B}_{\infty,\infty}^{-n/2}$ and $\dot{B}_{r,\widetilde{r}}^{s_{r}+\epsilon}\hookrightarrow\dot{B}_{\infty,\infty}^{-1+\epsilon}$, the energy estimate $\limsup_{t\nearrow T}{\|u(t)\|}_{L^{2}}<\infty$, the assumption $\lim_{t\nearrow T}{\|u(t)\|}_{\dot{B}_{\infty,\infty}^{-1/2}}=\infty$, and the interpolation inequality
\begin{equation}\label{interpbesonehalf}
    {\|u(t)\|}_{\dot{B}_{\infty,\infty}^{-1/2}} \leq {\|u(t)\|}_{\dot{B}_{\infty,\infty}^{-n/2}}^{\lambda}{\|u(t)\|}_{\dot{B}_{\infty,\infty}^{-1+\epsilon}}^{1-\lambda} \quad \text{for }\lambda=\frac{\epsilon-\frac{1}{2}}{\epsilon-1+\frac{n}{2}},
\end{equation}
we obtain the qualitative blowup estimate
\begin{equation}\label{qualitative-blowup}
    \lim_{t\nearrow T}{\|u(t)\|}_{\dot{B}_{r,\widetilde{r}}^{s_{r}+\epsilon}} = \infty.
\end{equation}
We will prove \eqref{u-blowup} by combining \eqref{qualitative-blowup} with the following ODE lemma.
\begin{lemma}\label{ode-lemma}
    (\cite{mccormick2016}, Lemma 2.1). If $\gamma,c>0$, $\partial_{t}X\leq cX^{1+\gamma}$, and $\lim_{t\nearrow T}X(t)=\infty$, then
    \begin{equation*}
        X(t) \geq {(\gamma c(T-t))}^{-1/\gamma} \quad \text{for all }t\in(0,T).
    \end{equation*}
\end{lemma}
Our goal is to derive a suitable differential inequality that allows us to apply Lemma \ref{ode-lemma}. We will achieve this by considering the antisymmetric tensor\footnote{In the case $n=3$, this is related to the vorticity vector $\overrightarrow{\omega}:=\nabla\times u$ by $\overrightarrow{\omega}={(\omega_{23},\omega_{31},\omega_{12})}^{T}$. One could view $\omega$ and $\overrightarrow{\omega}$ as being different ways of representing the exterior derivative of the 1-form $\sum_{i=1}^{n}u_{i}\,\mathrm{d}x^{i}$.}
\begin{equation}\label{w-from-u}
    \omega_{ij} := \nabla_{i}u_{j}-\nabla_{j}u_{i}.
\end{equation}
Since $u$ is divergence-free, we can express $u$ in terms of $\omega$ by the formula
\begin{equation}\label{u-from-w}
    u_{i} = {(-\Delta)}^{-1}\nabla_{j}\omega_{ij}.
\end{equation}
By Lemma \ref{useful-inequalities}, we deduce that \eqref{u-blowup} is equivalent to
\begin{equation}\label{w-blowup}
    {\|\omega(t)\|}_{\dot{B}_{r,\widetilde{r}}^{s_{r}+\epsilon-1}} \gtrsim_{\varphi,\epsilon,r} {(T-t)}^{-\epsilon/2},
\end{equation}
and that \eqref{qualitative-blowup} is equivalent to
\begin{equation}\label{w-qualitative-blowup}
    \lim_{t\nearrow T}{\|\omega(t)\|}_{\dot{B}_{r,\widetilde{r}}^{s_{r}+\epsilon-1}} = \infty.
\end{equation}
Applying the operator $X\mapsto\nabla_{i}X_{j}-\nabla_{j}X_{i}$ to the Navier-Stokes equations \eqref{intro-navier-stokes}, we see that $\omega$ satisfies
\begin{equation}\label{w-equations}
    \partial_{t}\omega_{ij}-\Delta\omega_{ij}+(u\cdot\nabla)\omega_{ij}+\omega_{ik}\nabla_{k}u_{j}=\omega_{jk}\nabla_{k}u_{i}.
\end{equation}
Applying $\dot{\Delta}_{J}$ to the equation \eqref{w-equations}, multiplying the result by ${|\dot{\Delta}_{J}\omega|}^{r-2}\dot{\Delta}_{J}\omega_{ij}$, summing over $i,j$ and integrating over $\mathbb{R}^{n}$, we obtain
\begin{equation}\label{w-energy}
\begin{aligned}
    &\frac{1}{r}\frac{\partial}{\partial t}\left({\|\dot{\Delta}_{J}\omega\|}_{L^{r}}^{r}\right) - \left\langle\Delta\dot{\Delta}_{J}\omega,{|\dot{\Delta}_{J}\omega|}^{r-2}\dot{\Delta}_{J}\omega\right\rangle \\
    &\quad = -\left\langle\dot{\Delta}_{J}((u\cdot\nabla)\omega),{|\dot{\Delta}_{J}\omega|}^{r-2}\dot{\Delta}_{J}\omega\right\rangle - \left\langle\dot{\Delta}_{J}(\omega_{ik}\nabla_{k}u_{j}-\omega_{jk}\nabla_{k}u_{i}),{|\dot{\Delta}_{J}\omega|}^{r-2}\dot{\Delta}_{J}\omega_{ij}\right\rangle.
\end{aligned}
\end{equation}
By the identities $\nabla\cdot u=0$, $\nabla_{k}\left({|\dot{\Delta}_{J}\omega|}^{r}\right)=r(\nabla_{k}\dot{\Delta}_{J}\omega_{ij}){|\dot{\Delta}_{J}\omega|}^{r-2}\dot{\Delta}_{J}\omega_{ij}$ and $\omega_{ij}=-\omega_{ji}$, we see that the right hand side of \eqref{w-energy} is equal to
\begin{equation*}
    \left\langle[u\cdot\nabla,\dot{\Delta}_{J}]\omega,{|\dot{\Delta}_{J}\omega|}^{r-2}\dot{\Delta}_{J}\omega\right\rangle - 2\left\langle\dot{\Delta}_{J}(\omega\cdot\nabla u),{|\dot{\Delta}_{J}\omega|}^{r-2}\dot{\Delta}_{J}\omega\right\rangle,
\end{equation*}
where we define ${(\omega\cdot\nabla u)}_{ij}:=\omega_{ik}\nabla_{k}u_{j}$. Writing $\Omega_{J}:=[u\cdot\nabla,\dot{\Delta}_{J}]\omega-2\dot{\Delta}_{J}(\omega\cdot\nabla u)$, and noting the inequality\footnote{Valid for $n\geq3$ and $r\in[2,\infty)$, proved in \cite[Lemmas 1-2]{robinson2014}.}
\begin{equation*}\label{lhs-bound}
    -\left\langle\Delta v,{|v|}^{r-2}v\right\rangle \gtrsim_{n,r} {\|v\|}_{L^{\frac{rn}{n-2}}}^{r},
\end{equation*}
we deduce that
\begin{equation}\label{w-energy-2}
    \frac{\partial}{\partial t}\left({\|\dot{\Delta}_{J}\omega\|}_{L^{r}}^{r}\right) + {\|\dot{\Delta}_{J}\omega\|}_{L^{\frac{rn}{n-2}}}^{r} \lesssim_{n,r} \left\langle\Omega_{J},{|\dot{\Delta}_{J}\omega|}^{r-2}\dot{\Delta}_{J}\omega\right\rangle.
\end{equation}

{\bf The case $\boldsymbol{\epsilon<2}$.} Suppose first that $\epsilon \in [1,2)$.  From \eqref{w-energy-2} we have
\begin{equation}\label{w-energy-3}
    \frac{\partial}{\partial t}\left({\|\omega\|}_{\dot{B}_{r,r}^{s_{r}+\epsilon-1}}^{r}\right) + {\|\omega\|}_{\dot{B}_{\frac{rn}{n-2},r}^{s_{r}+\epsilon-1}}^{r} \lesssim_{n,r} \sum_{J\in\mathbb{Z}}2^{Jr(s_{r}+\epsilon-1)}\left\langle\Omega_{J},{|\dot{\Delta}_{J}\omega|}^{r-2}\dot{\Delta}_{J}\omega\right\rangle.
\end{equation}
Since $\epsilon\in(0,2)$ and $r\in(1,\infty)$, the interval $I_{n,\epsilon,r}:=(\frac{2n}{r}-\frac{4}{r},\frac{2n}{r})\cap(\frac{2n}{r}-2+\epsilon,\frac{2n}{r}-\frac{2}{r}+\epsilon)$ is non-empty, so we are free to choose $r_{1}$ satisfying $\frac{2n}{r_{1}}\in I_{n,\epsilon,r}$. Let $r_{2}$ and $r_{3}$ be given by
\begin{equation*}\label{r2r3}
    -\frac{n}{r_{2}}=s_{r}+\epsilon-1-\frac{n}{r_{1}}, \quad r_{3}=\frac{r_{1}r_{2}}{r_{1}+r_{2}}.
\end{equation*}
Then $\frac{2n}{r}>\frac{2n}{r_{1}}>\frac{2n}{r}-\frac{4}{r}$ is equivalent to $r<r_{1}<\frac{rn}{n-2}$, while $\frac{2n}{r}-\frac{2}{r}+\epsilon>\frac{2n}{r_{1}}>\frac{2n}{r}-2+\epsilon$ is equivalent to $\frac{rn}{n+2(r-1)}<r_{3}<r$. Therefore
\begin{equation*}\label{r1range}
    {\left(\frac{r'n}{n-2}\right)}'=\frac{rn}{n+2(r-1)}<r_{3}<r<r_{1}<\frac{rn}{n-2}.
\end{equation*}
Writing $r_{4}=r_{3}'(r-1)$, by H\"{o}lder's inequality we have
\begin{equation}\label{rhs-bound}
\begin{aligned}
    \sum_{J\in\mathbb{Z}}2^{Jr(s_{r}+\epsilon-1)}\left\langle\Omega_{J},{|\dot{\Delta}_{J}\omega|}^{r-2}\dot{\Delta}_{J}\omega\right\rangle &\leq \sum_{J\in\mathbb{Z}}2^{Jr(s_{r}+\epsilon-1)}{\|\Omega_{J}\|}_{L^{r_{3}}}{\|\dot{\Delta}_{J}\omega\|}_{L^{r_{4}}}^{r-1} \\
    &\leq {\left\|J\mapsto2^{J(s_{r}+\epsilon-1)}{\|\Omega_{J}\|}_{L^{r_{3}}}\right\|}_{l^{r}}{\|\omega\|}_{\dot{B}_{r_{4},r}^{s_{r}+\epsilon-1}}^{r-1}.
\end{aligned}
\end{equation}
Since ${\left(\frac{r'n}{n-2}\right)}'<r_{3}<r$, it follows that $r'<r_{3}'<\frac{r'n}{n-2}$ and hence $r<r_{4}<\frac{rn}{n-2}$. By \eqref{interpolation-holder}, we deduce that
\begin{equation}\label{mu-interpolation}
    {\|\omega\|}_{\dot{B}_{r_{4},r}^{s_{r}+\epsilon-1}} \leq {\|\omega\|}_{\dot{B}_{r,r}^{s_{r}+\epsilon-1}}^{\mu}{\|\omega\|}_{\dot{B}_{\frac{rn}{n-2},r}^{s_{r}+\epsilon-1}}^{1-\mu} \quad \text{for }\mu=\frac{rn}{2}\left(\frac{1}{r_{4}}-\frac{n-2}{rn}\right).
\end{equation}
We now need to estimate ${\left\|J\mapsto2^{J(s_{r}+\epsilon-1)}{\|\Omega_{J}\|}_{L^{r_{3}}}\right\|}_{l^{r}}$. By the Bony estimates \eqref{bony-estimate-1} and \eqref{bony-estimate-3}, the inequality ${\|v\|}_{\dot{B}_{r_{2},\infty}^{0}}\lesssim_{\varphi}{\|v\|}_{L^{r_{2}}}$, and the assumption $r<\frac{n}{2-\epsilon}$ (which is equivalent to $s_{r}+\epsilon>1$), we have
\begin{equation}\label{product-estimate}
    {\|\omega\cdot\nabla u\|}_{\dot{B}_{r_{3},r}^{s_{r}+\epsilon-1}} \lesssim_{\varphi,\epsilon,r} {\|\omega\|}_{\dot{B}_{r_{1},r}^{s_{r}+\epsilon-1}}{\|\nabla u\|}_{L^{r_{2}}}+{\|\omega\|}_{L^{r_{2}}}{\|\nabla u\|}_{\dot{B}_{r_{1},r}^{s_{r}+\epsilon-1}}.
\end{equation}
On the other hand, by Proposition \ref{commutator-prop} we have $[u\cdot\nabla,\dot{\Delta}_{J}]\omega=\sum_{I=1}^{5}R_{J}^{I}$, where
\begin{equation}\label{commutator-estimate}
\begin{aligned}
    {\left\|J\mapsto2^{J(s_{r}+\epsilon-1)}{\|R_{J}^{1}\|}_{L^{r_{3}}}\right\|}_{l^{r}} &\lesssim_{\varphi,\epsilon,r,r_{1}} {\|\nabla u\|}_{L^{r_{2}}}{\|\omega\|}_{\dot{B}_{r_{1},r}^{s_{r}+\epsilon-1}}, \\
    {\left\|J\mapsto2^{J(s_{r}+\epsilon-1)}{\|R_{J}^{I}\|}_{L^{r_{3}}}\right\|}_{l^{r}} &\lesssim_{\varphi,\epsilon,r,r_{1}} {\|\nabla u\|}_{\dot{B}_{r_{1},r}^{s_{r}+\epsilon-1}}{\|\omega\|}_{\dot{B}_{r_{2},\infty}^{0}} \quad \text{for }I=2,3,4,5.
\end{aligned}
\end{equation}
Combining \eqref{product-estimate} and \eqref{commutator-estimate}, and noting the relations \eqref{w-from-u}-\eqref{u-from-w} and Lemma \ref{useful-inequalities}, we therefore have
\begin{equation}\label{W-estimate}
    {\left\|J\mapsto2^{J(s_{r}+\epsilon-1)}{\|\Omega_{J}\|}_{L^{r_{3}}}\right\|}_{l^{r}} \lesssim_{\varphi,\epsilon,r,r_{1}} {\|\omega\|}_{\dot{B}_{r_{1},r}^{s_{r}+\epsilon-1}\cap L^{r_{2}}}^{2}.
\end{equation}
The indices $r_{1},r_{2}$ we chosen to ensure that $r<r_{1}<\frac{rn}{n-2}$ and $0<s_{r}+\epsilon-1=\frac{n}{r_{1}}-\frac{n}{r_{2}}$, and that the spaces $\dot{B}_{r_{1},r}^{s_{r}+\epsilon-1}$ and $L^{r_{2}}$ have the same scaling; from these conditions we deduce the interpolation inequality
\begin{equation}\label{nu-interpolation}
    {\|\omega\|}_{\dot{B}_{r_{1},r}^{s_{r}+\epsilon-1}\cap L^{r_{2}}} \lesssim_{n,\epsilon,r,r_{1}} {\|\omega\|}_{\dot{B}_{r,r}^{s_{r}+\epsilon-1}}^{\nu}{\|\omega\|}_{\dot{B}_{\frac{rn}{n-2},r}^{s_{r}+\epsilon-1}}^{1-\nu} \quad \text{for }\nu=\frac{rn}{2}\left(\frac{1}{r_{1}}-\frac{n-2}{rn}\right),
\end{equation}
where the estimate on ${\|\omega\|}_{\dot{B}_{r_{1},r}^{s_{r}+\epsilon-1}}$ follows from \eqref{interpolation-holder}, while the estimate on ${\|\omega\|}_{L^{r_{2}}}$ is justified (writing $r_{5}=\frac{rn}{n-2}$ and $s_{r}+\epsilon-1=s_{r_{5}}+\epsilon-1+\frac{2}{r}$) by the embedding ${\|\omega\|}_{L^{r_{2}}}\leq{\|\omega\|}_{\dot{B}_{r_{2},1}^{0}}$, and the calculations
\begin{equation*}
    \text{If }r_{2}\geq r_{5}\text{ :}\quad{\|\omega\|}_{\dot{B}_{r_{2},1}^{0}}\lesssim_{n,\epsilon,r,r_{1}}{\|\omega\|}_{\dot{B}_{r_{2},\infty}^{s_{r_{2}}+\epsilon-1}}^{\nu}{\|\omega\|}_{\dot{B}_{r_{2},\infty}^{s_{r_{2}}+\epsilon-1+\frac{2}{r}}}^{1-\nu}\lesssim_{n}{\|\omega\|}_{\dot{B}_{r_{5},\infty}^{s_{r_{5}}+\epsilon-1}}^{\nu}{\|\omega\|}_{\dot{B}_{r_{5},\infty}^{s_{r_{5}}+\epsilon-1+\frac{2}{r}}}^{1-\nu}
\end{equation*}
\begin{equation*}
    \text{If }r_{2}< r_{5}\text{ :}\quad{\|\omega\|}_{\dot{B}_{r_{2},1}^{0}}\lesssim_{n,\epsilon,r,r_{1}}{\|\omega\|}_{\dot{B}_{r_{2},\infty}^{s_{r_{2}}+\epsilon-1}}^{\rho}{\|\omega\|}_{\dot{B}_{r_{2},\infty}^{s_{r}+\epsilon-1}}^{1-\rho}\lesssim_{n}{\|\omega\|}_{\dot{B}_{r,\infty}^{s_{r}+\epsilon-1}}^{\rho}{\|\omega\|}_{\dot{B}_{r,\infty}^{s_{r}+\epsilon-1}}^{(1-\rho)\sigma}{\|\omega\|}_{\dot{B}_{r_{5},\infty}^{s_{r}+\epsilon-1}}^{(1-\rho)(1-\sigma)}
\end{equation*}
for $\nu,\rho,\sigma\in(0,1)$ determined by \eqref{interpolation-holder}-\eqref{interpolation-geometric}. By the bounds \eqref{rhs-bound} and \eqref{W-estimate}, and the interpolation inequalities \eqref{mu-interpolation} and \eqref{nu-interpolation}, we obtain
\begin{equation*}
    \sum_{J\in\mathbb{Z}}2^{Jr(s_{r}+\epsilon-1)}\left\langle\Omega_{J},{|\dot{\Delta}_{J}\omega|}^{r-2}\dot{\Delta}_{J}\omega\right\rangle \lesssim_{\varphi,\epsilon,r} {\|\omega\|}_{\dot{B}_{r,r}^{s_{r}+\epsilon-1}}^{(r-1)\mu+2\nu}{\|\omega\|}_{\dot{B}_{\frac{rn}{n-2},r}^{s_{r}+\epsilon-1}}^{r+1-[(r-1)\mu+2\nu]},
\end{equation*}
where $\mu$ and $\nu$ are given by \eqref{mu-interpolation} and \eqref{nu-interpolation}.  Noting that
$$\frac{r-1}{r_{4}}+\frac{2}{r_{1}}=1-\frac{1}{r_{3}}+\frac{2}{r_{1}}= 1+\frac{1}{r_{1}}-\frac{1}{r_{2}}=1+\frac{1}{r}+\frac{\epsilon-2}n\, ,$$
we see that
\begin{equation*}
\begin{aligned}
    (r-1)\mu+2\nu &= \frac{rn}{2}\left((r-1)\left(\frac{1}{r_{4}}-\frac{n-2}{rn}\right)+2\left(\frac{1}{r_{1}}-\frac{n-2}{rn}\right)\right) \\
    &= \frac{rn}{2}\left(\frac{r-1}{r_{4}}+\frac{2}{r_{1}}\right)-\frac{(r+1)(n-2)}{2} \\
    &= 1+\frac{r\epsilon}{2}
\end{aligned}
\end{equation*}
and hence\footnote{As a side remark, we observe that if an estimate of the form $\sum_{J\in\mathbb{Z}}2^{Jr(s_{r}+\epsilon-1)}\left\langle\Omega_{J},{|\dot{\Delta}_{J}\omega|}^{r-2}\dot{\Delta}_{J}\omega\right\rangle \lesssim_{\varphi,\epsilon,r,\alpha,\beta} {\|\omega\|}_{\dot{B}_{r,r}^{s_{r}+\epsilon-1}}^{\alpha}{\|\omega\|}_{\dot{B}_{\frac{rn}{n-2},r}^{s_{r}+\epsilon-1}}^{\beta}$ holds for all antisymmetric $\omega$, then necessarily $\alpha=1+\frac{r\epsilon}{2}$ and $\beta=\frac{r}{2}(2-\epsilon)$. This observation can be justified by considering the effect of the rescaling $\omega\mapsto\kappa\omega$ and $x\mapsto 2^{N}x$ for $\kappa>0$ and $N\in\mathbb{Z}$.}
\begin{equation}\label{rhs-bound-2}
    \sum_{J\in\mathbb{Z}}2^{Jr(s_{r}+\epsilon-1)}\left\langle\Omega_{J},{|\dot{\Delta}_{J}\omega|}^{r-2}\dot{\Delta}_{J}\omega\right\rangle \lesssim_{\varphi,\epsilon,r} {\|\omega\|}_{\dot{B}_{r,r}^{s_{r}+\epsilon-1}}^{1+\frac{r\epsilon}{2}}{\|\omega\|}_{\dot{B}_{\frac{rn}{n-2},r}^{s_{r}+\epsilon-1}}^{\frac{r}{2}(2-\epsilon)}.
\end{equation}
By \eqref{w-energy-3}, \eqref{rhs-bound-2} and Young's product inequality, we deduce that
\begin{equation*}
    \frac{\partial}{\partial t}\left({\|\omega\|}_{\dot{B}_{r,r}^{s_{r}+\epsilon-1}}^{r}\right) \lesssim_{n,\epsilon,r} {\|\omega\|}_{\dot{B}_{r,r}^{s_{r}+\epsilon-1}}^{r+\frac{2}{\epsilon}}.
\end{equation*}
Applying Lemma \ref{ode-lemma} with $X(t)={\|\omega(t)\|}_{\dot{B}_{r,r}^{s_{r}+\epsilon-1}}^{r}$ and $\gamma=\frac{2}{r\epsilon}$, and noting \eqref{w-qualitative-blowup}, we conclude that \eqref{w-blowup}${}_{\epsilon<2}$ holds, which (as noted above) implies \eqref{eps-less-2}.

{\bf The case $\boldsymbol{\epsilon=2}$.} By \eqref{w-energy-2} and H\"{o}lder's inequality, for all $J,t$ satisfying $\dot{\Delta}_{J}\omega(t)\neq0$ in $\mathcal{S}'$ we have
\begin{equation}\label{eps2-w-energy}
    \frac{\partial}{\partial t}\left({\|\dot{\Delta}_{J}\omega\|}_{L^{r}}\right) \lesssim_{n,r} {\|\Omega_{J}\|}_{L^{r}}.
\end{equation}
If $\dot{\Delta}_{J}\omega(t_{0})=0$ in $\mathcal{S}'$, then either ${\left.\frac{\partial}{\partial t}\left({\|\dot{\Delta}_{J}\omega\|}_{L^{r}}\right)\right|}_{t=t_{0}}=0$ (in which case \eqref{eps2-w-energy} is true for $t=t_{0}$) or ${\left.\frac{\partial}{\partial t}\left({\|\dot{\Delta}_{J}\omega\|}_{L^{r}}\right)\right|}_{t=t_{0}}\neq0$ (in which case \eqref{eps2-w-energy} is true for $t$ close to $t_{0}$, so by continuity it is true for $t=t_{0}$). Therefore \eqref{eps2-w-energy} holds for all $J\in\mathbb{Z}$ and $t\in(0,T)$, so we can estimate
\begin{equation}\label{eps2-w-energy-2}
    \frac{\partial}{\partial t}\left({\|\omega\|}_{\dot{B}_{r,1}^{s_{r}+1}}\right) \lesssim_{n,r} {\left\|J\mapsto2^{J(s_{r}+1)}{\|\Omega_{J}\|}_{L^{r}}\right\|}_{l^{1}}.
\end{equation}
We now need to estimate ${\left\|J\mapsto2^{J(s_{r}+1)}{\|\Omega_{J}\|}_{L^{r}}\right\|}_{l^{1}}$. By the Bony estimates \eqref{bony-estimate-1} and \eqref{bony-estimate-3}, the inequality ${\|v\|}_{\dot{B}_{\infty,\infty}^{0}}\lesssim_{\varphi}{\|v\|}_{L^{\infty}}$, and the assumption $r<\infty$ (which is equivalent to $s_{r}+1>0$), we have
\begin{equation}\label{eps2-prod}
    {\|\omega\cdot\nabla u\|}_{\dot{B}_{r,1}^{s_{r}+1}} \lesssim_{\varphi,r} {\|\omega\|}_{\dot{B}_{r,1}^{s_{r}+1}}{\|\nabla u\|}_{L^{\infty}} + {\|\omega\|}_{L^{\infty}}{\|\nabla u\|}_{\dot{B}_{r,1}^{s_{r}+1}}.
\end{equation}
On the other hand, by Proposition \ref{commutator-prop} we have $[u\cdot\nabla,\dot{\Delta}_{J}]\omega=\sum_{I=1}^{5}R_{J}^{I}$, where
\begin{equation}\label{eps2-comm}
\begin{aligned}
    {\left\|J\mapsto2^{J(s_{r}+1)}{\|R_{J}^{1}\|}_{L^{r}}\right\|}_{l^{1}} &\lesssim_{\varphi,r} {\|\nabla u\|}_{L^{\infty}}{\|\omega\|}_{\dot{B}_{r,1}^{s_{r}+1}}, \\
    {\left\|J\mapsto2^{J(s_{r}+1)}{\|R_{J}^{I}\|}_{L^{r}}\right\|}_{l^{1}} &\lesssim_{\varphi,r} {\|\nabla u\|}_{\dot{B}_{r,1}^{s_{r}+1}}{\|\omega\|}_{\dot{B}_{\infty,\infty}^{0}} \quad \text{for }I=2,3,4,5.
\end{aligned}
\end{equation}
Combining \eqref{eps2-prod} and \eqref{eps2-comm}, and noting the relations \eqref{w-from-u}-\eqref{u-from-w} and Lemma \ref{useful-inequalities}, we therefore have
\begin{equation}\label{eps2-W-estimate}
    {\left\|J\mapsto2^{J(s_{r}+1)}{\|\Omega_{J}\|}_{L^{r}}\right\|}_{l^{1}} \lesssim_{\varphi,r} {\|\omega\|}_{\dot{B}_{r,1}^{s_{r}+1}}^{2}.
\end{equation}
By \eqref{eps2-w-energy-2} and \eqref{eps2-W-estimate} we have
\begin{equation*}
    \frac{\partial}{\partial t}\left({\|\omega\|}_{\dot{B}_{r,1}^{s_{r}+1}}\right) \lesssim_{\varphi,r} {\|\omega\|}_{\dot{B}_{r,1}^{s_{r}+1}}^{2}.
\end{equation*}
Hence if $\epsilon =2$, applying Lemma \ref{ode-lemma} with $X(t)={\|\omega\|}_{\dot{B}_{r,1}^{s_{r}+1}}$ and $\gamma=1=\frac{2}{\epsilon}$, and noting \eqref{w-qualitative-blowup}${}_{\epsilon=2}$, we conclude that \eqref{w-blowup}${}_{\epsilon=2}$ holds, which implies \eqref{eps-is-2} and completes the proof of Theorem \ref{main-theorem}.
\end{proof}%%%

\end{document}